\documentclass[12pt,a4paper,leqno,twoside]{amsart}
\usepackage{latexsym}
\usepackage{amsmath}
\usepackage{amsfonts}
\usepackage{amssymb}

\input xy
\xyoption{all}

\numberwithin{equation}{section}

\theoremstyle{plain}
\newtheorem{thm}{Theorem}[section]

\newtheorem{lem}[thm]{Lemma}

\theoremstyle{definition}

\newcommand{\ZZ}{\mathbb{Z}}

\newcommand{\HH}{\mathbb{H}}
\newcommand{\QQ}{\mathbb{Q}}

\newcommand{\OO}{\mathbb{O}}
\newcommand{\FF}{\mathbb{F}}

\newcommand{\Map}{\mathrm{Map}}

\begin{document}

\title{Loop homology of quaternionic projective spaces}

\author[M.~\v Cadek, Z.~Moravec ]{Martin \v Cadek,
Zden\v ek  Moravec}

\address{\newline Department of Mathematics, Masaryk University,
Kotl\' a\v rsk\' a 2, 611 37 Brno, Czech Republic}
\email{cadek@math.muni.cz}

\email{106681@mail.muni.cz}

\date{April 9, 2010}


\thanks{This work was supported by the grant MSM0021622409 of
the Czech Ministry of Education and the grant 0964/2009 of Masaryk
University.}

\abstract
{We determine the Batalin-Vilkovisky algebra structure of the integral loop
homology of quaternionic projective spaces and octonionic projective plane. }
\endabstract

\maketitle

\section{Introduction}

Let $M$ be a closed oriented manifold of dimension $d$ and let 
$LM = \Map (S^{1}, M)$ denote its free loop space. By loop homology 
we understand the homology groups of $LM$ with the degree shifted by $-d$
$$
\HH _{*} (LM) = H_{*+d}(LM).
$$
In \cite{ChS} it was shown that this graded group can be equipped with 
a product and an operator $\Delta$ giving $\HH_{*}(LM)$ the structure of a
Batalin-Vilkovisky algebra. The methods computing the product on concrete 
manifolds are based either on the modified Serre spectral sequence derived 
in \cite{CJY} or on the isomorphism of the loop homology of $M$ with
the Hochschild cohomology $HH^*(C^*(M);C^*(M))$ of the cochain complex as
rings, \cite{CJ}. There is also a way of defining a BV-structure on
$HH^*(C^*(M);C^*(M))$, \cite{Tr}. The BV-algebra structures 
on the loop homology and the Hochschild cohomology are isomorphic over the 
fields of characteristic zero (\cite{FT}) but not over other coefficients 
in general. Hence the  computation of the BV operator is more subtle. So far 
the BV-algebra structure of the loop homology with integral coefficients has 
been determined for the Lie groups \cite{He1}, for the spheres \cite{M}, 
for the complex Stiefel manifolds \cite{Ta} and for the complex projective 
spaces \cite{He2}. Over rationals it has been described for the quaternionic 
projective spaces \cite{Y} and the surfaces \cite{V}.  

The aim of this note is to describe the BV-algebra structure of the integral
loop homology of the quaternionic projective spaces $\HH P^{n}$ and the 
octonionic projective plane $\OO P^2$. 

\begin{thm}\label{main1} The string topology BV-algebra structure of 
$\HH P^n$ is given by
$$
\HH_*(L\HH P^n;\ZZ)\cong 
\frac{\ZZ[a,b,x]}{\langle a^{n+1}, b^2, a^{n}\cdot b, (n+1) a^{n}\cdot
x\rangle}
$$
with $a \in \HH_{-4}(L\HH P^{n};\ZZ)$,  $b \in \HH_{-1} (L\HH P^{n};\ZZ)$ 
and  $x \in \HH_{4n+2}(L\HH P^{n})$, and
$$
\Delta(a^{p} x^{q}) =0, \quad \Delta(a^{p} b x^{q} ) = [(n-p) + q (n + 1) ]
a^{p} x^{q} 
$$
for all $0 \leq p \leq n$, $0 \leq q$.
\end{thm}

Let us note that for $n=1$ the quaternionic projective space is $S^4$ and
the statement agrees with the result obtain by L. Menichi in \cite{M} for
even dimensional spheres.

\begin{thm}\label{main2} There are elements $a \in \HH_{-8}(L\OO P^{2};\ZZ)$,  $b \in
\HH_{-1} (L\OO P^{2};\ZZ)$
and  $x \in \HH_{22}(L\OO P^{2})$ such that the string topology BV-algebra 
structure of $\OO P^2$ is given
by
$$
\HH_*(L\OO P^2;\ZZ)\cong
\frac{\ZZ[a,b,x]}{\langle a^{3}, b^2, a^{2}\cdot b, 3 a^{n}\cdot
x\rangle}
$$
and
$$
\Delta(a^{p} x^{q}) =0, \quad \Delta(a^{p} b x^{q} ) = (2+3q-p)a^{p} x^{q}
$$
for all $0 \leq p \leq 2$, $0 \leq q$.
\end{thm}

The statements of Theorem \ref{main1} and \ref{main2} concerning the ring
structure are consequences of the computation of $HH^*(\ZZ[y]/y^{n+1};
\ZZ[y]/y^{n+1})$ in \cite{Y} and the ring isomorhism between the loop
homology and the Hochschild cohomology. 
Nevertheless, we provide an alternative proof using the Serre spectral sequence for the fibrations
$\Omega M\to LM\to M$ converging to the ring $\HH_*(LM;\ZZ)$. These computations 
will be carried out in the next section.

In the last section we will show what the BV operator $\Delta$ looks like. 
We use the knowledge of $\Delta$ on $S^4$ and $S^8$ and the inclusions 
$S^4=\HH P^1\hookrightarrow \HH P^n$ and $S^8\hookrightarrow \OO P^2$. 
The computation will be completed by comparing $\Delta$ in integral homology
with BV-operator $\Delta$ in rational cohomology computed by Yang in
\cite{Y}. The results show that for the quaternionic projective 
spaces and the octonionic projective plane the BV-algebra structures on the 
loop homology and the Hochschild homology over integers are isomorphic 
(in contrast to the complex projective spaces, see \cite{He2}).

\section{The ring structure of loop homology}

According to \cite{CJY} the spectral sequence for the fibration
$\Omega M\to LM\to M$ with $ E_{p,q}^{2} = H^{-p}(M ; H_{q}(\Omega M;\ZZ))$
and the product coming from the Pontryagin product in $H_*(\Omega M;\ZZ)$ and 
the cup product in $H^*(M;H_*(\Omega M;\ZZ)$  converges to $\HH_{p+q}(LM;\ZZ)$ 
as an algebra. To apply this spectral sequence to $M=\HH P^n$ we have to
determine the Pontryagin ring $H_*(\Omega \HH P^n;\ZZ)$. We will consider
$n\ge 2$ since for $\HH P^1=S^4$ the statement of Theorem \ref{main1} has
been proved in \cite{M}. 

\begin{lem}\label{omegahp} For $n\ge 2$ the Pontryagin ring structure of 
$H_*(\Omega \HH P^n;\ZZ)$ is given by
$$
H_*(\Omega\HH P^n;\ZZ)\cong
\ZZ[x]\otimes \Lambda[t]
$$
where the degrees of $x$ and $t$ are $4n+2$ and $3$, respectively.
\end{lem}
 
\begin{proof}
The Hopf fibration $ S^3\to S^{4n+3} \to \HH P^{n} $ gives us
the fibration
\begin{equation}
 \Omega S^{4n+3} \xrightarrow{j}   \Omega \HH P^{n} \xrightarrow{p}  S^3
\end{equation}
Since $p_{*} :\pi_{k} (\Omega \HH P^{n}) \to \pi_{k} (S^{3})$ is an
isomorphism for $0\le k\le 6$, there is up to homotopy a unique map 
$i: S^{3} \to \Omega \HH P^{n}$ such that $p\circ i$ is homotopic to the 
identity on $S^3$.
Therefore the long exact sequence of homotopy groups for this fibration
passes to short exact sequences which split:
$$
\xymatrix{
0 \ar[r] & \pi_{*}(\Omega S^{4n+3}) \ar[r]^-{j_{*} } &  \pi_{*}(\Omega \HH
P^{n}) \ar[r]^-{p_{*}} &  \pi_{*}(S^3) \ar[r]  \ar@<1ex>@{.>}[l]^-{i_*}& 0
}
$$
Denote by $\mu$ the Pontryagin product on $\Omega \HH P^{n}$. The map 
$h=\mu \circ (j,i):\Omega S^{4n+2}\times S^3\to \Omega\HH P^n$ 
is a homotopy equivalence since it induces an isomorphism of homotopy
groups. So we obtain an isomorphism of homology groups
$$
H_*(\Omega\HH P^n;\ZZ)\cong H_*(\Omega S^{4n+3};\ZZ)\otimes
H_*(S^3;\ZZ)\cong \ZZ[x]\otimes \Lambda[t].
$$
The Pontryagin ring structure of $H_*(\Omega\HH P^n;\ZZ)$ can be recovered
using the duality  between the Hopf algebras $H_*(\Omega\HH P^n;\ZZ)$ and
$H^*(\Omega\HH P^n;\ZZ)$. The map $h$ induces an algebra
isomorphism $h^*:H^*(\Omega \HH P^n;\ZZ)\to H^*(\Omega S^{4n+3};\ZZ)\otimes
H^*(S^3;\ZZ)$. We know that $H^*(\Omega\HH P^n;\ZZ)$ is a commutative
associative Hopf algebra with $\mu^*$ as a coproduct. As an algebra
$H^*(\Omega\HH P^n;\ZZ)\cong \Gamma_{\ZZ}[\alpha_1,\alpha_2,\dots]\otimes 
\Lambda[\beta]$, where $\Gamma_{\ZZ}[\alpha_1,\alpha_2,\dots]$ is a divided
polynomial algebra with generators $\alpha_i$ and relations 
$\alpha_{i} \alpha_{j} = \binom{i+j}{i}\alpha_{i+j}$. 
Since $j^*:H^*(\Omega\HH P^n;\ZZ)\to H^*(\Omega S^{4n+3};\ZZ)$ is a 
homomorphism of Hopf algebras and  the Hopf algebra structure of 
$H^*(\Omega S^{4n+3};\ZZ)$ is well known, the coproduct on 
$H_*(\Omega\HH P^n;\ZZ)$ is given by
\begin{gather*}
\mu^*(\beta) = \beta \otimes 1 + 1 \otimes \beta, \quad 
\mu^*(\alpha_{k})= \sum_{k=i+j} \alpha_{i} \otimes \alpha_{j}, \\
\mu^*(\alpha_{k}\beta) = \sum_{k=i+j} \alpha_{i} \beta \otimes \alpha_{j}
+ \sum_{k=i+j} \alpha_{i} \otimes \beta \alpha_{j}.
\end{gather*}
By duality this coproduct completely determines the Potryagin product in
$H^*(\Omega\HH P^n;\ZZ)$. Let $t \in H_{*}(\Omega \HH P^{n})$ be a dual 
element to $\beta$, $x_{k}$ be a dual to $\alpha_{k}$ and $z_{k}$ be a dual to 
$\alpha_{k} \beta$. Then 
$$
x_{i+j}=x_{i}x_{j}, \quad z_{i+j} = z_{i} x_{j}.
$$ 
If we put $x = x_{1}$, we obtain $x_i=x^i$ and $z_{i}=x^it$ for all
$i\ge 0$. This completes the proof.
\end{proof}

Now we return to the spectral sequence  
converging to the algebra $\HH_*(L\HH P^n;\ZZ)$. Its $E^2$ term is
$$
E_{-p,q}^{2}=H^{p}(\HH P^n;H_q(\Omega\HH P^n;\ZZ)) \cong H^{p}(\HH P^{n};\ZZ) 
\otimes H_{q} (\Omega \HH P^{n};\ZZ)\cong  
\frac{\ZZ[a] \otimes \ZZ[x,t]}{\langle a^{n+1}, t^2\rangle}
$$
where $a\in H^4(M;\ZZ)$ and $x$, $t$ as in Lemma \ref{omegahp}.
The stages $E^4$ and $E^{4n}$ of the spectral sequence with possible nonzero 
differentials are shown in the following diagram: 
\begin{displaymath}
\xymatrix@R=12pt@C=16pt@M=0pt{
 & & & & & &  &  &  & H_{q} (\Omega \HH P^{n}) \\
E_{p,q}^{4}, E_{p,q}^{4n} & & & & & &  &  &  & q\\
 & & & & & &  &  &  & \\
 & &\bullet \ar@{--}[l] \ar@{--}[rrrr]& & & & \bullet \ar@{--}[r]& \bullet \ar@{--}[r]& \bullet \ar@{-}[d] \ar@{->}[uu] \ar@{}[r]|(.2){x^{2}t} & 8n + 7 \\ 
 & &\bullet \ar@{--}[l] \ar@{--}[rrrr]& & & & \bullet \ar@{--}[r]& \bullet \ar@{--}[r] \ar[lu]|{d^{4}} & \bullet \ar@{-}[dddd] \ar[lu]|{d^{4}} \ar@{}[r]|(.2){x^{2}} & 8n + 4 \\
 & & & & & &  &  &  & \\
 & & & & & &  &  &  & \\
 & & & & & &  &  &  & \\
 & &\bullet \ar@{--}[l] \ar@{--}[rrrr]& & & & \bullet \ar@{--}[r]& \bullet \ar@{--}[r]& \bullet \ar@{-}[d] \ar[lllllluuuu]_-{d^{4n}} \ar@{}[r]|(.2){xt}& 4n+5 \\ 
 & &\bullet \ar@{--}[l] \ar@{--}[rrrr]& & & & \bullet \ar@{--}[r]& \bullet \ar@{--}[r]\ar[lu]|{d^{4}}& \bullet \ar@{-}[dddd] \ar[lu]|{d^{4}} \ar@{}[r]|(.2){x}& 4n+2 \\
 & & & & & &  &  &  & \\
 & & & & & &  &  &  & \\
 & & & & & &  &  &  & \\
 & &\bullet \ar@{--}[l] \ar@{--}[rrrr] & & & & \bullet \ar@{--}[r]& \bullet \ar@{}[u]|(.5){b} \ar@{--}[r]& \bullet \ar@{-}[d] \ar[lllllluuuu]_-{d^{4n}} \ar@{}[r]|(.2){t} & 3 \\ 
 & &\bullet \ar@{}[d]|(.4){a^{n} } \ar@{->}[l] \ar@{-}[rrrr]  & \ar[lu]|{d^{4}}& & & \bullet \ar@{}[d]|(.4){a^{2}} \ar@{-}[r] \ar[lu]|{d^{4}}& \bullet \ar@{}[d]|(.4){a} \ar@{-}[r] \ar[lu]|{d^{4}}& \bullet \ar[lu]|{d^{4}}          & 0 \\
H^{p}(\HH P^{n}) &p &-4n & & & &  -8 &  -4 & 0 & \\
}
\end{displaymath}
Since $E_{p,q}^{\infty}\Rightarrow \HH_{p+q} (L \HH P^{n};\ZZ) 
= H _{p+q + 4n} (L \HH P^{n};\ZZ)$ we can determine the differentials 
from the knowledge of the additive structure of $H_{*}(L \HH P^{n};\ZZ)$.

To compute it we use the result of \cite{BO} on the existence of a stable 
decomposition
$$
(L \HH P^{n})_{+} \simeq \HH P^{n}_{+} \vee \bigvee_{l \geq 1} 
S(\eta)^{ l \xi \oplus (l-1) \zeta }
$$
where $\eta$ is tangent bundle of the quaternionic projective space 
$\HH P^{n}$, $\xi$ is the $3$-dimensional Lie algebra bundle over 
$\HH P^{n}$ and $\zeta$ is the fibrewise tangent bundle 
of $S(\eta)$ and $S(\eta)^{\omega}$ stands for the Thom space of the vector 
bundle $\omega$ over $S(\eta)$.
Note that $\dim S(\eta) = 8n-1$ and $ \dim \zeta = 4n-1$. Using the Gysin 
long exact sequence for the fibration $ S^{4n-1} \to S(\eta) \to \HH P^{n}$ 
and the fact that the Euler characterictic class of $\eta$ is an
$(n+1)$-multiple of the generator $a^n\in H^{4n}(\HH P^n;\ZZ)$ we get 
\begin{displaymath}
H_{i} S(\eta) = \left\{
\begin{array}{ll} 
\ZZ & \qquad i = 0,4,\dots,4n-4, 4n+3, 4n+7, \dots, 8n-1, \\
\ZZ_{n+1}& \qquad i = 4n-1, \\
0& \qquad  {\rm otherwise}. \\
\end{array} \right.
\end{displaymath}
The dimesion of the vector bundle $l \xi \oplus (l-1) \zeta$ 
is $4n(l - 1)+2l+1$, so due to the Thom isomorphism  
$$
H_{*} (L \HH P^{n};\ZZ) \cong H_{*}(\HH P^{n};\ZZ) \oplus 
\bigoplus_{l\ge 1}H_{*+4n(l-1)+2l+1}(S (\eta);\ZZ).
$$
Since $\HH_*(L\HH P^n;\ZZ)\cong E^{\infty}_{*,*}$, the $E^{\infty}$ stage 
of the spectral sequence is the following
\begin{displaymath}
\xymatrix@R=12pt@C=16pt@M=1pt{
& &  & & &  &                     &  &  & q\\
 & &                                      & & &                       & 
 &  &  & \\
 & &\ZZ \ar@{--}[l] \ar@{--}[r]        &\ZZ \ar@{--}[rrr]& &             
 & \ZZ \ar@{--}[r]     & \ZZ \ar@{--}[r] & 0 \ar@{-}[d] \ar@{->}[uu]& 8n + 7 \\ 
 & &\ZZ_{n+1} \ar@{--}[l] \ar@{--}[r]    &\ZZ \ar@{--}[rrr]& &  
 & \ZZ \ar@{--}[r]     & \ZZ \ar@{--}[r] 
& \ZZ \ar@{-}[dddd] 
& 8n + 4 \\
 & &                                      & &   &                     &  
 &  &  & \\
 & &                                      & &   &                     &
 &  &  & \\
 & &                                      & &   &                     & 
 &  &  & \\
 & &\ZZ \ar@{--}[l] \ar@{--}[r]        & \ZZ \ar@{--}[rrr]& & 
 & \ZZ \ar@{--}[r]     & \ZZ \ar@{--}[r] & 
0 \ar@{-}[d] 
& 4n+5 \\ 
 & &\ZZ_{n+1} \ar@{--}[l] \ar@{--}[r]    & \ZZ \ar@{--}[rrr]& &  
 & \ZZ \ar@{--}[r]     & \ZZ \ar@{--}[r] 
& \ZZ \ar@{-}[dddd] 
& 4n+2 \\
 & &                                      & & &                       &
 &  &  & \\
 & &                                      & & &                       & 
 &  &  & \\
 & &                                      & & &                       & 
 &  &  & \\
 & &\ZZ \ar@{--}[l] \ar@{--}[r]        & \ZZ \ar@{--}[rrr]& &   
 & \ZZ \ar@{--}[r]     & \ZZ \ar@{--}[r]& 
0 \ar@{-}[d] 
& 3 \\ 
 & &\ZZ \ar@{->}[l] \ar@{-}[r]         & 
\ZZ 
\ar@{-}[rrr]& &    & \ZZ \ar@{-}[r] 
& \ZZ \ar@{-}[r] 
& \ZZ 
& 0 \\
 &p &-4n & & & &  -8 &  -4 & 0 & \\
}
\end{displaymath}
It forces the differentials 
$d^{4}$ in $E^4$ to be zero and the differentials 
$d^{4n}:E^{4n}_{0,(4n+2)i+3}\to E^4_{-4n,(4n+2)(i+1)}$ to be the 
multiplication by $n+1$. 

So $E^{\infty}_{*,*}$ as a ring is generated by the group generators 
$a \in E_{-4,0}^{\infty}\cong H^{4}(\HH P^{n};\ZZ)$, 
$x \in E_{0,4n + 2}^{\infty}\cong H_{4n+2}(\Omega \HH P ^{n};\ZZ)$ and
$b \in E_{-4,3}^{\infty}\cong
H^{4}(\HH P ^{n};\ZZ) \otimes H_{3}(\Omega \HH P ^{n};\ZZ)$ which satisfy
relations $a^{n+1} = 0$, $(n+1) \ x \otimes a^{n} = 0$, 
$b \otimes a^{n} = 0$, $b^{2} = 0$. We conclude that as rings
$$
\HH_{*} (L \HH P^{n};\ZZ) \cong E_{*,*}^{\infty} \cong \frac{\ZZ[a,b,x]}
{\langle a^{n+1}, b^2, a^{n}b, (n+1) a^{n}x \rangle}.
$$

In the case of the octonionic projective plane the derivation of the ring
structure of the loop homology follows the same lines.

\begin{lem}\label{omegaop} The Pontryagin ring structure of
$H_*(\Omega \OO P^2;\ZZ)$ is given by
$$
H_*(\Omega\OO P^2;\ZZ)\cong
\ZZ[x]\otimes \Lambda[t]
$$
where $|x|=22$ and $|t|=7$.
\end{lem}

\begin{proof} Using the fact that
$$H^{*}(\Omega \OO P^{2}) \cong \Gamma_{\ZZ}[\alpha_1,\alpha_2,\dots]
\otimes \Lambda[\beta]$$
where $|\alpha_i|=22i$ and $|\beta|=7$, proved in \cite{Ha},
we can proceed in the same way as in the proof of Lemma \ref{omegahp}.
\end{proof}

The additive structure of $H_{*}(L \OO P^{2};\ZZ)$ was found in \cite{BO} 
using a stable decomposition of $L\OO P^2$  derived there: 
\begin{displaymath}
H_{i} (L \OO P^{2}) = \left\{
\begin{array}{ll} 
\ZZ & \qquad i = 0, 8, 16, 22m-15, 22m-7, 22m+8, 22m+16, \\
\ZZ_{3}& \qquad i = 22m, \\
0& \qquad  {\rm otherwise}. \\   
\end{array} \right.
\end{displaymath}  

It yields that in the spectral sequence starting with 
$$
E^2_{-p,q}=H^p(\OO P^2;H_q(\Omega\OO P^2;\ZZ))\cong H^p(\OO P^2;\ZZ)
\otimes H_q(\Omega\OO P^2;\ZZ)\cong \frac{\ZZ[a]\otimes \ZZ[x,b]}
{\langle a^3, b^2\rangle}
$$
all the differentials are zero with the exception of the differentials 
$d^{16}:E^{16}_{0,22m-15}\to E^{16}_{-16,22m}$ which act as the multiplication 
by $3$. The group generators $a\in E^{\infty}_{-8,0}\cong H^*(\OO P^2;\ZZ)$, 
$x \in E_{0,22}^{\infty}\cong H_{22}(\Omega \OO P ^{2};\ZZ)$ and
$b \in E_{-8,7}^{\infty}\cong
H^{8}(\HH P ^{n};\ZZ) \otimes H_{7}(\Omega \OO P ^{2};\ZZ)$, 
generate  $E^{\infty}_{*,*}\cong \HH_*(L\OO
P^2;\ZZ)$ as a ring satisfying relations $a^3=0$, $b^2=0$, $3ax=0$ 
and $a^2b=0$.

\section{The BV operator}

The BV operator $\Delta:\HH_*(LM)\to \HH_{*+1}(LM)$ and its unshifted
version $\Delta':H_*(LM)\to H_{*+1}(LM)$ come from the canonical
action of $S^1$ on $LM$. So any map $f:N\to M$ between manifolds induces 
a homomorhism $H_*(LN)\to H_*(LM)$ which commutes with $\Delta'$.
To determine the BV operator on $\HH_*(L\HH P^n;\ZZ)$ and 
$\HH_*(L\OO P^2;\ZZ)$ we use this fact for the inclusions $S^4\hookrightarrow
\HH P^n$ and $S^8\hookrightarrow \OO P^2$ together with the knowledge of the BV
operator on $\HH_*(S^n;\ZZ)$, see \cite{M}.

We start with the quaternionic projective space. First, $\Delta (a^px^q)=0$
because $\HH_{|a^px^q|+1}(L\HH P^n;\ZZ)=0$. Since $\HH_{|a^pb|+1}(L\HH
P^n;\ZZ)\cong\ZZ$ is generated by $a^p$, there is an integer $\nu_p$ such that
$\Delta(a^pb)=\nu_p a^p$. Due to the relation 
\begin{align*}
\Delta(xyz) =& \Delta(xy)z + (-1)^{|x|}x \Delta(yz)+(-1)^{(|x|-1)|y|}y 
\Delta(xz) \\
& - \Delta(x)yz - (-1)^{|x|}x \Delta(y) z - (-1)^{|x|+|y|}xy\Delta(z) \\
\end{align*}
we obtain 
$$
\Delta(a^{p} b) =  \Delta(a^{p-1} a b) = a^{p-1} \Delta(a b) + 
a \Delta(a^{p-1} b) - a^{p} \Delta( b).
$$
It yields the equation $\nu_{p} = \nu_{1} - \nu_{0} + \nu_{p-1}$, which can 
be rewritten as 
$$
\nu_{p} = p (\nu_{1} - \nu_{0}) + \nu_{0}.
$$
The relation $a^{n}b = 0$ implies that $\nu_{n}= 0$. Consequently, for $p=n$
the equation above gives $n \nu_{1} = (n-1) \nu_{0}$. 
Hence for $n\ge 2$ the only possible integer solutions of this equation are  
$$
\nu_{1} = (n-1) \lambda_{n}, \quad \nu_{0} = n \lambda_{n},
$$
where $\lambda_{n}$ is an integer. Consequently, 
we obtain $\nu_{p} = (n-p)\lambda_{n}$.

For $n=1$ the quaternionic projective space is $S^4$. According to
\cite{M} the generators of $\HH_*(L\HH P^1;\ZZ)$ as an algebra are 
$a_1$, $b_1$ and $v_1$  in degrees $-4$, $-1$ and $6$,
respectively, and $\Delta(b_1)=1$.

The standard inclusion $i:S^4=\HH P^{1} \hookrightarrow \HH P^{n}$ 
induces  the commutative diagram of fibrations 
$$
\xymatrix{
\Omega \HH P^1  \ar[rr] \ar[d] &&  \Omega \HH P^n \ar[d] \\
L \HH P^1 \ar[rr] \ar[d] && L \HH P^{n} \ar[d] \\
\HH P^1 \ar[rr] && \HH P^{n} \\
}
$$
The inclusion $i$ induces an isomorhism $H_{4} (\HH P^1;\ZZ) \cong 
H_{4} (\HH P^{n};\ZZ)$ and the inclusion $\Omega \HH P^{1} \hookrightarrow 
\Omega \HH P^{n}$
yields an  isomorphism $H_{3} (\Omega \HH P^1;\ZZ) \cong H_{3} (\Omega \HH
P^{n};\ZZ)$.

The commutative diagram above gives us a homomorphism between the 
Serre spectral sequences of the corresponding fibrations. (Here we consider
the spectral sequences with $E^2_{p,q}=H_p(M;H_q(\Omega M;\ZZ))$.) This
homomorphism is an isomorphism on  $E^2_{4,0}$ and $E^2_{0,3}$
terms and it remains an isomorphism also on $E^{\infty}_{4,0}$ and 
$E^{\infty}_{0,3}$. Consequently, for $i=-1$ and $0$ we obtain
$$
\HH_{i}(LS^4;\ZZ)=H_{i+4}(LS^4;\ZZ)\cong H_{i+4}(L\HH P^n;\ZZ)=\HH_{i-4(n-1)}
(L\HH P^n;\ZZ).
$$
Choose $b\in \HH_1(L\HH P^n;\ZZ)$ and $a\in \HH_4(L\HH P^n;\ZZ)$ so that
$a^{n-1}$ is the image of $1$ and $a^{n-1}b$ is the image of 
$b_1\in \HH_1(LS^4;\ZZ)$ under the above isomorphisms. 
Since these isomorphisms commute with $\Delta$, we obtain
$\Delta(a^{n-1}b)=a^{n-1}$. Consequently, $\lambda_n=1$. 

Analogously we get  $\Delta(bx^q)=\rho_qx^q$ for an
integer $\rho_q$ and derive that
$$
\rho_{q} =  q (\rho_{1}-\rho_{0}) + \rho_{0}.
$$
Since $\rho_{0} = \nu_{0} = n \lambda_{n}=n$,  we obtain
\begin{align*}
\Delta(a^{p} b x^{q}) &= \Delta(a^{p} b) x^{q} + a^{p}\Delta(b x^{q}) - a^{p} x^{q} \Delta( b ) = \\
 &= [(n-p) + q (\rho_1 - n )] a^{p} x^{q}.
\end{align*}

In \cite{Y} T. Yang computed the BV-algebra structure of the Hochshild
cohomology of truncated polynomials. Using Theorem 1 from \cite{FT} on the
existence of a BV-algebra isomorhism between the loop homology $\HH_*(LM;\FF)$ 
of a manifold and the Hochschild cohomology $HH^*(C^*(M);C^*(M))$ of the 
singular cochain complex over fields of characteristic zero, he was able to 
calculate the BV-algebra structure of $\HH_*(L\HH P^n;\QQ)$. 
This is given by 
$$
\HH_{*} (L \HH P^{n};\QQ) = \frac{\QQ[\alpha, \beta ,\chi]}
{\langle \alpha^{n+1}, \beta^{2}, \alpha^{n} \beta, \alpha^{n} \chi\rangle},
$$
where $|\alpha|=-4$, $|\beta|=-1$, $|\chi|=4n+2$,
and by
$$
\Delta(\alpha^{p} \chi^{q}) =0, \quad  
\Delta(\alpha^{p} \beta \chi^{q} ) = [(n-p) + q (n+1) ] \alpha^{p} \chi^{q}.
$$
Consider the homomorphism $r_{*}: \HH_{*} (L \HH P^{n}; \ZZ) 
\to \HH_{*} (L \HH P^{n}; \QQ)$  induced by the inclusion
$\ZZ\hookrightarrow \QQ$. Let 
$$
r_*(a)=k\alpha, \quad r_*(b)=l\beta, \quad r_*(x)=m\chi, 
$$
where, $k,l,m \in \QQ - \{ 0 \}$. Since $r_*$ is a homomorphism of
BV-algebras, we obtain
$$
[(n-p) + q (\rho_1 -n)] k^{p} m^{q} 
\alpha^{p} \chi^{q} =r_*(\Delta(a^pbx^q)=\Delta(r_*(a^pbx^q)) 
= l[(n-p) + q (n+1)] k^{p} m^{q} \alpha^{p} \chi^{q}.
$$
Putting $q=0$ we get $l=1$. Then the choice $p=0$, $q=1$ yields 
$\rho_1 = 2n+1$ which concludes our computation.

\medskip
To compute the BV operator in $\HH_*(L\OO P^2;\ZZ)$ we can follow the same 
procedure step by step replacing the inclusion 
$S^4\hookrightarrow \HH P^n$ by the inclusion $S^8 \hookrightarrow \OO P^2$.

\end{document}